\documentclass[12pt]{amsart}%
\usepackage{amsfonts}
\usepackage{amsmath}
\usepackage{amssymb}
\usepackage{graphicx}%
\providecommand{\U}[1]{\protect\rule{.1in}{.1in}}
\newtheorem{theorem}{Theorem}
\theoremstyle{plain}

\newtheorem{corollary}{Corollary}

\newtheorem{lemma}{Lemma}

\newtheorem{proposition}{Proposition}

\numberwithin{equation}{section}
\begin{document}
\title[Multiplicities for $\theta$--group harmonics]{An analogue of the Kostant-Rallis multiplicity theorem for $\theta$--group harmonics}
\author{Nolan R. Wallach}
\dedicatory{To Roger Howe with admiration.}
\begin{abstract}
The main result in this paper is the generalization of the Kostant-Rallis
multiplicity formula to general $\theta$--groups (in the sense of Vinberg).
The special cases of the two most interesting examples one for E$_{6}$ (three
qubits) and one for E$_{8}$ are given explicit formulas.

\end{abstract}
\maketitle

\section{Introduction}

The purpose of this paper is to give proofs of analogues for a Vinberg
$\theta$--group of two results of Kostant-Rallis \cite{Kostant-Rallis} for the
case when  $\theta$ is an involution of a semi--simple Lie algebra,
$\mathfrak{g}$, over $\mathbb{C}$. To describe the results we need some
notation. Let $V$ denote the $-1$ eigenspace of $\theta$. Set $H$ equal to the
identity component of the centralizer of $\theta$ in the automorphism group of
$\mathfrak{g}$. We use the notation $\mathcal{O}(V)$ for the algebra of
polynomials on $V$ and $\mathcal{O}(V)^{H}$ for the algebra of $H$-invariants.
Finally we set $\mathcal{H}$ equal to the $H$--module of harmonics. The first
result is that $\mathcal{O}(V)$ is isomorphic with $\mathcal{O}(V)^{H}%
\otimes\mathcal{H}$ as a $H$ and $\mathcal{O}(V)^{H}$ module. The second
result the generalization of their formula for the $H$--multiplicities in
$\mathcal{H}$. If $\theta$ an automorphism of $\mathfrak{g}$ of order
$0<m<\infty$ then a \textquotedblleft$\theta$--group\textquotedblright\ is a
pair $(H,V)$ where  $V$ is the eigenspace for a principal $m$--th root of
unity in $\mathfrak{g}$ and $H$ is the identity component of the centralizer of
$\theta$ in the automorphism group of $\mathfrak{g}$ restricted  to $V$. We prefer to use
the term Vinberg pair for $(H,V)$. We will also give a proof of Vinberg's main
theorem (that says that $\mathcal{O}(V)^{H}$ is a polynomial algebra over
$\mathbb{C}$) for these pairs that does not rely on classification based on a
brilliant theorem of Panyushev \cite{Panyushev}. The reasons for this
inclusion are that we have not seen this proof in full detail in the literature. Also the other results of Vinberg
(which do not depend on classification) that are used in our argument are also
needed in the proof of the multiplicity formula. Vinberg's original
argument is complicated whereas both Panyushev's argument and the reduction of
Vinberg's theorem to it are not complicated.  Also in light of the recent deep applications of Vinberg theory in characteristic $p$ we note that the
arguments should extend to arbitrary fields where the Shephard-Todd theorem
applies using deeper \'{e}tale theoretic arguments. Vladimir Popov has communicated to us that 
he and Vinberg published a similar argument in their (much) earlier paper [PV]

Vinberg's paper \cite{Vinberg} was a major addition to the literature of
geometric invariant theory. Even if the reader is not interested in the
results that go beyond \cite{Vinberg}, the listing of the main results of that
paper (with explicit references to the original) might be reason enough to
read this one. The interested reader should also look at the tables in
\cite{Vinberg}. An expanded exposition of the results in of Vinberg quoted and
the newer results in this paper will appear as part of my forthcoming book
\cite{Wallach}.

The paper is organized as follows: Section 2 describes the part of Vinberg's
work that is necessary for our proof of his main theorem. Section 3. is a
description of the Shephard-Todd theorem, Panyushev's result and a proof of
Vinberg's main theorem. Section 4 studies maximal compact subgroups of $H$ and
a description of the Kempf-Ness set for a Vinberg pair. Section 5 contains
several more results of Vinberg and our proof of the freeness theorem above
and the multiplicity formula. I give two important examples of the
multiplicity formula in an effectively computable form. The first is what is
probably the most interesting non-symmetric example for E$_{6}$ (what the
physicists call \textquotedblleft three qutrits\textquotedblright)\ and the
second is the example that was studied extensively by Vinberg and Elashvili
\cite{Elesh-Vin} for E$_{8}$.

We thank Hanspeter Kraft for his patient explanation of Panyushev's result;
Vladimir Popov for pointing out his work with Vinberg [PV] and
the referee for this article for pointing out the paper of Dodak and Kac on what
they call polar representations \cite{Dodak-Kac}. It appears that many of the
preliminary results in this paper are also true in this larger context. Thus
it is likely that variant of the multiplicity formula is true in this context.
Jeb Willenbring's student, Alexander Heaton, will be doing part of his thesis
work on this problem.

This paper is dedicated to my long time friend Roger Howe on the occasion of
his seventieth birthday. We met in Berkeley in 1966 when I was a first year
postdoc and he was a graduate student. I was asked by Cal Moore to teach the
third quarter of his Lie groups course. I was awed by the quality of the
students. In that exceptional group Roger stood out. Although we have written
only one joint paper (joint also with Tom Enright) we have had many deep
mathematical conversations. It was my good fortune that Roger's parents spent
their retirement in San Diego. This led Roger to visit UCSD often. I miss
Roger's parents and I miss his visits.

\section{\label{Resuklts-1}Definitions and some of Vinberg's results}

A $\theta$--group is a pair $(H,V)$ of a finite dimensional vector space over
$\mathbb{C}$, $V$, and a Zariski closed, connected, reductive subgroup of
$GL(V)$ constructed as follows: $\mathfrak{g}$ is a semi--simple Lie algebra
over $\mathbb{C}$, $\theta$ is an automorphism of finite order, $m,$ of
$\mathfrak{g}$, $\zeta$ is a primitive $m$--th root of unity, $V$ is the
$\zeta$ eigenspace for $\theta$, $L$ is the connected subgroup of
$\mathrm{Aut}(\mathfrak{g)}$ with $\mathrm{Lie}(L)=\mathfrak{g}^{\theta}$
(eigenspace for $1$) and $H=L_{|V}$. In this paper we will call $(H,V)$ a
Vinberg pair.

Vinberg's theory reduces the study of the orbit structure and invariant theory
of a Vinberg pair to the case when $\mathfrak{g}$ is simple. In this paper we
will concentrate on this case. The purpose of this section is to give a
listing of the results of Vinberg that we will need in our proof of his main
theorem (all of which are proved without case by case checks). First some
general notation.

If $U$ is a finite dimensional vector space over $\mathbb{C}$ let
$\mathcal{O}(U)$ denote the polynomial functions on $U$ and if $X$ is Zariski
closed in $U$ then $\mathcal{O}(X)=\mathcal{O}(U)_{|X}$ (the regular functions
on $U$). If $G$ is a algebraic group acting on $X$ regularly then we have a
representation of $G$ on $\mathcal{O}(X)$ by%
\[
(gf)(x)=f(g^{-1}x),g\in G,x\in X.
\]
We note that if $f\in\mathcal{O}(X)$ then the span of $Gf$, $Z$, is finite
dimensional and the corresponding action of $G$ on $Z$ is regular. We set
$\mathcal{O}(X)^{G}=\{f\in\mathcal{O}(X)|gf=f\}$. This algebra is finitely
generated over $\mathbb{C}$ and so we can form the maximal spectrum, $X//G$,
of $\mathcal{O}(X)^{G}$ which is an affine variety with $\mathcal{O}%
(X//G)=\mathcal{O}(X)^{G}$. Null cone of $X$ is the set $\{x\in
X|f(x)=f(0),f\in\mathcal{O}(X)^{G}\}$.

Let $(H,V)$ be a Vinberg pair that corresponds to a simple Lie algebra
$\mathfrak{g}$ and automorphism $\theta$. Here are the results

\noindent1. Let $v\in V$. Then $Hv$ is closed in $V$ if and only if $v$ is
semi-simple in $\mathfrak{g}$ \cite{Vinberg}, Proposition 3.\smallskip

\noindent2. $v$ is an element of the null cone of $V$ if and only if $v$ is
nilpotent in $\mathfrak{g}$ \cite{Vinberg}, Proposition 1.\smallskip

A Cartan subspace of $V$ is a subspace, $\mathfrak{a}$, such that

a) every element of $\mathfrak{a}$ is semi-simple in $\mathfrak{g}$,

b) $[\mathfrak{a},\mathfrak{a}]=0$ and

c) $\mathfrak{a}$ is maximal with respect to a) and b).

\noindent3. All Cartan subspaces are conjugate under $H$. Define the common
dimension of the Cartan subspaces to be the rank of the Vinberg pair.
\cite{Vinberg} Theorem 1.\smallskip

\noindent4. If $\mathfrak{a}$ is a Cartan subspace then $H\mathfrak{a}$ is the
union of the closed orbits of $H$ [V] Corollary to Theorem 1.\smallskip

\noindent5. Set for $v\in V$ the variety $X_{v}=\{x\in V|f(x)=f(v),f\in
\mathcal{O}(V)^{H}\}$ ($X_{0}$ is the null cone). If $l$ is the rank of
$(H,V)$ then
\[
\dim X_{v}=\dim V-l.
\]
Furthermore every irreducible component of $X_{v}$ contains an open $H$-orbit,
(indeed, $X_{v}$ is a finite union of $H$--orbits) \cite{Vinberg} Theorems 4,5.

\noindent6. Let $\mathfrak{a}$ be a Cartan subspace of $V$ and let
$N_{H}(\mathfrak{a})=\{h\in H|h\mathfrak{a=a\}}$ and $W(\mathfrak{a}%
)=N_{H}(\mathfrak{a})_{|\mathfrak{a}}$. Then $W_{H}(\mathfrak{a})$ is a finite
group and $V//H$ is isomorphic with $\mathfrak{a}/W_{H}(\mathfrak{a})$ as an
affine variety. \cite{Vinberg} Theorem 7.

\section{Complex reflections, Panyushev's result and Vinberg's main theorem}

Let $U$ be a finite dimensional vector space over $\mathbb{C}$ then a complex
reflection on $U$ is a linear isomorphism of finite order such that $\dim
\ker(U-I)=\dim U-1$. Shephard-Todd in \cite{Shephard-Todd} proved

\begin{theorem}
Let $U$ be a finite dimensional vector space and $G$ a finite group acting on
$U$. Then $G$ is generated by complex reflections if and only if $U/G$ is
isomorphic with $U$ as an affine variety.
\end{theorem}

The \textquotedblleft only if\textquotedblright\ part of this theorem is
usually stated

\begin{theorem}
If $U$ is a finite dimensional vector space and $G\subset GL(U)$ is a finite
subgroup generated by complex reflections then $\mathcal{O}(U)^{G}$ is
generated as an algebra over $\mathbb{C}$ by $\dim U$ algebraically
independent homogenous polynomials.
\end{theorem}

The proof in \cite{Shephard-Todd} of this part of the theorem was by a case by
case check. The \textquotedblleft if\textquotedblright\ part is proved without
case by case checking by reducing to the \textquotedblleft only
if\textquotedblright\ part. Chevalley \cite{Chevalley} gave a proof of this
part without classification under the hypothesis that $G$ is generated by
reflections of order $2$. The literature seems unanimous that Serre pointed
out to him that his proof of the special case proved the full result without
any real change.

The theorem of Panyushev \cite{Panyushev} rests on the full theorem above. In
\cite{Wallach} a complete exposition of the proof of this theorem and the
Shephard-Todd theorem is given.

\begin{theorem}
Let $V$ and $U$ be finite dimensional complex vector spaces and let $H$ be a
reductive group acting on $V$ regularly and let $W\subset GL(U)$ be a finite
subgroup. Let $p:V\rightarrow V//H$ be the natural surjection. Assume that if
$X\subset V//H$ is Zariski closed and of codimension at least $2$ then
$p^{-1}(X)$ is of codimension at least $2$ in $V$. If $V//H$ is isomorphic
with $U/W$ as an affine variety then $W$ is generated by complex reflections.
\end{theorem}

Panyushev's proof of this theorem is an ingenious application of the
Shephard-Todd theorem. We note that it is an exercise to prove that if $H$ is
semi-simple then the codimension assumption is satisfied. Thus 6. in the
previous section implies that if $H$ is semi-simple that $W_{H}(\mathfrak{a})
$ is generated by reflections. This contains all of the cases Vinberg looks at
in his tables for the exceptional groups.

The main theorem in \cite{Vinberg} says

\begin{theorem}
Let $(H,V)$ be a Vinberg pair and let $\mathfrak{a}$ be a Cartan subspace of
$V.$Then $W_{H}(\mathfrak{a})$ is generated by complex reflections of
$\mathfrak{a}$.
\end{theorem}

\begin{proof}
If $\dim\mathfrak{a}=1$ then any map of of finite order of $\mathfrak{a}$ is a
complex reflection.

We may now assume that $\dim\mathfrak{a}\geq2$. Let $X\subset V//H$ be Zariski
closed and of codimension at least $2$ we show that $p^{-1}(X)$ has
codimension at least $2.$ Let $l=\dim\mathfrak{a}$. Then $\dim X=l-k$ with
$k\geq2$. Let $Y\subset p^{-1}(X)$ be an irreducible component. Then since $H$
is connected $Y$ is $H$--invariant. We may assume that $\overline{p(Y)}=X$
(this might increase $k$). If $x\in p(Y)$ then $\dim\left(  p_{|Y}\right)
^{-1}(x)\leq\dim p^{-1}(x)=\dim V-l$. The theorem of the fiber (c.f.
\cite{GoodWall}) implies that there exists $x\in p(Y)$ such that
\[
\dim\left(  p_{|Y}\right)  ^{-1}(x)=\dim Y-\dim X.
\]
Thus
\[
\dim V-l\geq\dim Y-\dim X=\dim Y-(l-k).
\]
So
\[
\dim Y\leq\dim V-l+l-k=\dim V-k.
\]
Thus the hypothesis of Panyushev's theorem is satisfied in this case. This
completes our proof of Vinberg's main theorem.
\end{proof}

\section{Maximal compact subgroups}

If $\mathfrak{g}$ is a semi-simple Lie algebra over $\mathbb{C}$ then we may
realize $\mathfrak{g}$ as a Lie sub-algebra of $M_{n}(\mathbb{C})$ with the
property that if $x\in\mathfrak{g}$ then $x^{\ast}=\bar{x}^{T}\in\mathfrak{g}$
(i.e. the conjugate transpose is in $\mathfrak{g}$). On $\mathfrak{g}$ we put
the Hilbert space structure
\[
\left\langle x,y\right\rangle =\mathrm{tr}(\mathrm{ad}x\mathrm{ad}y^{\ast
})=B(x,y^{\ast})
\]
with $B$ the Killing form. $Aut(\mathfrak{g})$ is closed under adjoint with
respect to $\left\langle ...,...\right\rangle $.

\subsection{$\theta$ can be assumed unitary}

\begin{lemma}
$\mathrm{Aut}(\mathfrak{g})$ is closed under adjoint with respect to
$\left\langle ...,...\right\rangle $.
\end{lemma}

\begin{proof}
Let $g\in\mathrm{Aut}(\mathfrak{g})$. Then%
\[
\left\langle gx,y\right\rangle =B(gx,y^{\ast})=B(x,g^{-1}y^{\ast})
\]%
\[
=B(x,\left(  (g^{-1}y^{\ast})^{\ast}\right)  ^{\ast})=\left\langle
x,(g^{-1}y^{\ast})^{\ast}\right\rangle .
\]
Thus setting $\sigma(x)=x^{\ast}$, the adjoint of $g$ is $\sigma g^{-1}\sigma
$. We assert that this element is in $\mathrm{Aut}(\mathfrak{g}) $. To see
this we calculate%
\[
\lbrack\sigma g^{-1}\sigma x,\sigma g^{-1}\sigma y]=-[g^{-1}\sigma
x,g^{-1}\sigma y]^{\ast}=
\]%
\[
=-(g^{-1}[x^{\ast},y^{\ast}])^{\ast}=(g^{-1}[x,y]^{\ast})^{\ast}=\sigma
g^{-1}\sigma\lbrack x,y]\text{.}
\]

\end{proof}

Now let $\theta$ be an automorphism of order $m<\infty$ of $\mathfrak{g}$ and
let $(H,V)$ be the corresponding Vinberg pair ($V=\mathfrak{g}_{\zeta}$,
$\zeta$ a primitive $m$--th root of $1$. Let $G=Aut(\mathfrak{g})$ and let
$G^{o}$ denote the identity component of $G$. Let $U$ be the unitary group of
$\mathfrak{g}$ relative to $\left\langle ...,...\right\rangle $. Then $G\cap
U$ is a maximal compact subgroup of $G$ and $G^{o}\cap U$ is maximal compact
in $G^{o}$.

The conjugacy theorem of maximal compact subgroups implies that there exists
$g\in G$ such that $g\theta g^{-1}$ is contained in $G\cap U$. Thus we have proved

\begin{theorem}
There exists $g\in G$ such that $g\theta g^{-1}$ normalizes $G^{o}\cap U$.
Furthermore replacing $\theta$ with $g\theta g^{-1}$then $\mathfrak{g}_{\zeta
}^{\ast}=\mathfrak{g}_{\zeta^{-1}}$.
\end{theorem}

Replace $\theta$ with $g\theta g^{-1}$. We note that if $L$ is the connected
subgroup of $G^{o}$ corresponding to $\mathfrak{g}^{\theta}$ then $L^{\ast}%
=L$. This implies

\begin{lemma}
$H$ is invariant under the adjoint corresponding to the restriction of
$\left\langle ...,...\right\rangle $ to $V$. Thus in particular $U(V)\cap H$
is a maximal compact subgroup of $H$.
\end{lemma}

\subsection{Kempf-Ness theory}

We recall a bit of the Kempf-Ness theory. Let $V$ be a finite dimensional
complex Hilbert space and let $H\subset GL(V)$ be a Zariski closed subspace
invariant under adjoint. Let $K=H\cap U(V)$ then $K$ is a maximal compact
subgroup of $H$ and $H$ is the Zariski closure of $K$. We say that $v\in V$ is
critical if $\left\langle Xv,v\right\rangle =0$ for all $X\in\mathrm{Lie}(H)$.
We will use the notation $Crit(V)$ for the space of critical elements of $V$
(this set is usually called the Kempf-Ness set) . The Kempf-Ness theorem
\cite{Kempf-Ness} says

\begin{theorem}
Notation as above

1. $x\in V$ is critical if and only if $\left\Vert hx\right\Vert
\geq\left\Vert x\right\Vert $ for all $h\in H$.

2. If $x\in V$ is critical then $\{y\in Hx|\left\Vert y\right\Vert =\left\Vert
x\right\Vert \}=Kx$.

3. If $x\in V$ and $Hx$ is closed then $Hx$ contains a critical element.

4. If $x\in V$ is critical then $Hx$ is closed.
\end{theorem}

The hard part of this theorem is part 4. We will now apply this to $(H,V)$
which we can assume satisfies the hypotheses of the theorem in light of the
material in the last sub-section. We now carry over $\mathfrak{g,\theta
,}\left\langle ...,...\right\rangle ,H,V=\mathfrak{g}_{\zeta}$. Set
$K_{H}=H\cap V$.

\begin{lemma}
$Crit(V)=\{x\in V|[x,x^{\ast}]=0\}$.
\end{lemma}

\begin{proof}
$x\in Crit(V)$ if and only if $\left\langle Xx,x\right\rangle =0$ for all
$X\in\mathrm{Lie}(H)$. This condition is if and only if $B(Xx,x^{\ast})=0$ for
all $X\in\mathrm{Lie}(H)$. Hence $x\in Crit(V)$ if and only if
$B([X,x],x^{\ast})=B([x,x^{\ast}],X)=0$ for all $X\in\mathfrak{g}^{\theta} $.
But our assumptions imply that $[x,x^{\ast}]\in\mathfrak{g}^{\theta}$.
\end{proof}

\section{Analogues of the Kostant-Rallis theorems}

Let $(H,V)$ be a Vinberg pair corresponding (as above) to a simple Lie algebra
$\mathfrak{g}$ with an automorphism of order $m$, $\theta$. Let $\mathfrak{a}$
be a Cartan subspace of $V$. We assume, as we may, that $\mathfrak{g}\subset
M_{n}(\mathbb{C})$ is invariant under adjoint and thus we have the inner
product
\[
\left\langle x,y\right\rangle =B(x,y^{\ast}).
\]
The restriction of this form yielding an inner product on $V$.

\subsection{The freeness}

Let $K_{H}=H\cap U(V)$. Then $K_{H}$ is a maximal compact subgroup of $H$. We
set $V_{1}=\{v\in V|\left\langle v,\mathfrak{a}\right\rangle =0\}.$If
$p_{\mathfrak{a}}$ and $p_{1}$ are the natural projections of respectively $V$
to $\mathfrak{a}$ and $V$ to $V_{1}$ we will identify $\mathcal{O}%
(\mathfrak{a)}$ and $\mathcal{O}(V_{1})$ respectively with $p_{\mathfrak{a}%
}^{\ast}(\mathcal{O}(\mathfrak{a}))$ and $p_{1}^{\ast}(\mathcal{O}(V_{1}))$
thus we have the graded algebra isomorphism%
\[
\mathcal{O}(\mathfrak{a})\otimes\mathcal{O}(V_{1})\rightarrow\mathcal{O}(V)
\]
under multiplication. Set $W=W_{H}(\mathfrak{a})$ then using Th\'{e}or\`{e}me
4 ii p.115 of \cite{Bourbaki} there is a subspace $\mathcal{A}$ of
$\mathcal{O(\mathfrak{a})}$ such that the map
\[
\mathcal{O}(\mathfrak{a})^{W}\otimes\mathcal{A}\rightarrow\mathcal{O}%
(\mathfrak{a})
\]
given by multiplication is a graded isomorphism and the Shephard Todd theorem
implies that we can take $\mathcal{A}$ to be a graded subspace of
$\mathcal{O}(\mathfrak{a})$ of dimension equal to $|W|$ and is a $W$ module
equivalent to the regular representation. The next result is analogous to
Lemma 12.4.11 in \cite{GoodWall}.

\begin{proposition}
The map $\mathcal{O}(V)^{H}\otimes\mathcal{A}\otimes\mathcal{O}(V_{1}%
)\rightarrow\mathcal{O}(V)$ given by multiplication is a graded vector space isomorphism.
\end{proposition}

\begin{proof}
We have seen that the restriction map $p_{\mathfrak{a}}^{\ast}:\mathcal{O}%
(V)^{H}\rightarrow\mathcal{O}(\mathfrak{a})^{W}$ is a graded algebra
isomorphism. Thus if we grade the tensor products above by the tensor product
grade then the graded components of $\mathcal{O}(V)^{L}\otimes\mathcal{A}%
\otimes\mathcal{O}(V_{1})$ and $\mathcal{O}(\mathfrak{a})^{W}\otimes
\mathcal{A}\otimes\mathcal{O}(V_{1})$ have the same dimension. Now the rest of
the argument is identical to that of Lemma 12.4.11 \cite{GoodWall}.
\end{proof}

The following result is proved in exactly the same way as in the last
paragraph of p.602 in [GW] by induction on the degree.

\begin{corollary}
\label{Harmonics}We extend $\left\langle ...,...\right\rangle $ to an inner
product on $\mathcal{O}(V)$ and define $\mathcal{H}^{j}=\left(  \left(
\mathcal{O}(V)\mathcal{O}_{+}(V)^{H}\right)  ^{j}\right)  ^{\perp}$ in
$\mathcal{O}(V)^{j}$ relative to this inner product. Then $\mathcal{H=\oplus
}_{j=0}^{\infty}\mathcal{H}^{j}$ is an $H$--module isomorphic with
$\mathcal{O}(V)/\left(  \mathcal{O}(V)\mathcal{O}_{+}(V)^{H}\right)  $ and
furthermore the map%
\[
\mathcal{O}(V)^{H}\otimes\mathcal{H\rightarrow O}(V)
\]
given by multiplication is a linear bijection.
\end{corollary}

We note that the ideal $\mathcal{O}(V)\mathcal{O}_{+}(V)^{H}$ defines the null
cone of $V$. Thus if we could show that the ideal $\mathcal{O}(V)\mathcal{O}%
_{+}(V)^{H}$ is a radical ideal then $\mathcal{H}$ could be identified with
$\mathcal{O}(\mathcal{N})$ with $\mathcal{N}$ the null cone of $V$. This is
one of the main results of Kostant and Rallis in the case when $\theta$ is an
involution. A result of Panyushev (c.f. \cite{Kraft-Schwarz}) proves that this
ideal is reduced if $H$ is semi--simple. The technique of Kostant-Rallis
\cite{Kostant-Rallis} does not work in the context of Vinberg pairs. However
their multiplicity theorem does generalize as does the technique used in
\cite{GoodWall} to prove the theorem.

\subsection{A few more results of Vinberg}

If $\lambda\in\mathfrak{a}^{\ast}$ set
\[
\mathfrak{g}^{\lambda}=\{x\in\mathfrak{g}|[h,x]=\lambda(h)x,h\in
\mathfrak{a}\}
\]
and $\Sigma(\mathfrak{a})=\{\lambda\neq0|\mathfrak{g}^{\lambda}\neq0\}$. We
set
\[
\mathfrak{a}^{\prime}=\{h\in\mathfrak{a}|\lambda(h)\neq0,\lambda\in
\Sigma(\mathfrak{a})\}.
\]

Set $C_{\mathfrak{g}}(\mathfrak{a})=\mathfrak{g}^{0}$, $C_{\mathfrak{g}%
}(h)=\ker\mathrm{ad}h$ and $C_{H}(\mathfrak{a})=\{g\in H|gh=h,h\in
\mathfrak{a}\}$. If $h\in\mathfrak{a}^{\prime}$ then $C_{\mathfrak{g}%
}(h)=C_{\mathfrak{g}}(\mathfrak{a})$. The following results are contained in
subsection 2 of section 3 of \cite{Vinberg}

\begin{theorem}
$C_{\mathfrak{g}}(\mathfrak{a})\cap V=\mathfrak{a\oplus n}$ with
$\mathfrak{n}$ a subspace of $V$ consisting of nilpotent elements.
$(C_{H}(\mathfrak{a})_{|\mathfrak{n}},\mathfrak{n})$ is a Vinberg pair of rank
$0$.
\end{theorem}

The second part of the theorem combined with 5. in Section \ref{Resuklts-1} implies

\begin{corollary}
The space $\mathfrak{n}$ is a finite union of $C_{H}(\mathfrak{a})$ orbits.
\end{corollary}

We also note that it implies

\begin{corollary}
If $h\in\mathfrak{a}^{\prime}$ then $X_{h}$ contains a unique open $H$--orbit,
$H(h+x)$ where $C_{H}(\mathfrak{a})x$ is the unique open $C_{H}(\mathfrak{a})$
orbit in $\mathfrak{n}$.
\end{corollary}

\begin{corollary}
If $h\in\mathfrak{a}^{\prime}$ then $X_{h}=H(h+\mathfrak{n})$, in particular,
$X_{h}$ is irreducible.
\end{corollary}

Another result, of a different nature, that will be used in the next
subsection is the content of subsection 1 of section 3 in \cite{Vinberg}.

\begin{theorem}
Let $T_{\mathfrak{a}}$ be the intersection of all Zariski closed subgroups of
$G$ whose Lie algebra contains $\mathfrak{a}$. Then $T_{\mathfrak{a}}$ is a
torus that is the center of the group $C_{G}(\mathfrak{a})$. If $\mathfrak{t}%
_{\mathfrak{a}}=\mathrm{Lie}(T_{\mathfrak{a}})$ then
\[
\mathfrak{t}_{\mathfrak{a}}=\oplus_{%
\begin{array}
[c]{c}%
1\leq j<m\\
\gcd(j,m)=1
\end{array}
}\mathfrak{t}_{\mathfrak{a}}\cap\mathfrak{g}_{\zeta^{j}}%
\]
and each space $\mathfrak{t}_{\mathfrak{a}}\cap\mathfrak{g}_{\zeta^{j}}$ is a
Cartan subspace of the Vinberg pair $(L_{|\mathfrak{g}_{\zeta^{j}}%
},\mathfrak{g}_{\zeta^{j}})$.
\end{theorem}

\subsection{The critical set revisited}

Let $\mathfrak{a}$ be a Cartan subspace of $V$ ($=\mathfrak{g}_{\zeta}$). Let
$x\in\mathfrak{a}^{\prime}$. Since $Hx$ is closed there exists $y\in
Crit(V)\cap Hx$. Write $y=gx.$ We replace $x$ with $y$ and $\mathfrak{a}$ with
$g\mathfrak{a}$. Thus we may assume that $x\in\mathfrak{a}^{\prime}$ is
critical. Since $x$ is critical we have $[x,x^{\ast}]=0$. Noting that
\textrm{$Lie$}$(T_{\mathfrak{a}})_{\zeta^{-1}}$ is a Cartan subspace for the
Vinberg pair $(L_{|\mathfrak{g}_{\zeta^{-1}}},\mathfrak{g}_{\zeta^{-1}})$ we
see that $C_{\mathfrak{g}}(\mathfrak{a})_{\zeta^{-1}}=\mathrm{Lie}%
(T_{\mathfrak{a}})_{\zeta^{-1}}\oplus\mathfrak{u}$ with $\mathfrak{u}$
consisting of nilpotent elements. This implies that $x^{\ast}\in
\mathrm{Lie}(T_{\mathfrak{a}})\cap\mathfrak{g}_{\zeta^{-1}}$. Hence
Lie$(T_{\mathfrak{a}})\cap\mathfrak{g}_{\zeta^{-1}}$ is contained in the set
of semi-simple elements in the centralizer of $x^{\ast}$ in $\mathfrak{g}%
_{\zeta^{-1}}$ which is $\mathfrak{a}^{\ast}$. Recalling that $\dim
\mathrm{Lie}(T_{\mathfrak{a}})\cap\mathfrak{g}_{\zeta^{-1}}=\dim
\mathfrak{a}=\dim\mathfrak{a}^{\ast}$. We have proved

\begin{proposition}
We may choose a Cartan sub-algebra, $\mathfrak{a}\subset V$ such that
$[\mathfrak{a},\mathfrak{a}^{\ast}]=0$.
\end{proposition}

\begin{proposition}
$Crit(V)=K_{H}\mathfrak{a}$.
\end{proposition}

\begin{proof}
The above lemma implies that $\mathfrak{a}\subset Crit(V)$. Suppose that $x\in
Crit(V)$ then $Hx$ is closed. Thus there exists $g\in H$ such that
$gx=y\in\mathfrak{a}$. Thus $\left\Vert x\right\Vert =\left\Vert y\right\Vert
$ (since both are critical). Hence there exists $k\in K_{H}$ such that $ky=x$
by the Kempf-Ness theorem..
\end{proof}

\begin{proposition}
If $w\in W_{H}(\mathfrak{a})$ then there exists $k\in K_{H}$ such that
$k_{|\mathfrak{a}}=w$.
\end{proposition}

\begin{proof}
Let $x\in\mathfrak{a}$ be such that if $\lambda,\mu\in\Sigma(\mathfrak{a}%
)\cup\{0\}$ then $\lambda(x)=\mu(x)$ implies $\lambda=\mu$. Such an
$x\in\mathfrak{a}$ exists. Indeed, define
\[
S=\{\lambda-\mu|\lambda,\mu\in\Sigma(\mathfrak{a})\cup\{0\},\lambda\neq\mu\}
\]
then $S$ is a finite set in $\mathfrak{a}^{\ast}$ (here the super script means
dual space) and $x$ is an element in $\mathfrak{a}$ such that $\xi(x)\neq0$
for $\xi\in S$. Let $h\in H$ be such that $h_{|\mathfrak{a}}=w$. Then $hx\in
Hx\cap Crit(V)=K_{H}x$. So $hx=kx$ for some $k\in K_{H}$. Now $w^{\ast}%
\Sigma(\mathfrak{a})=\Sigma(\mathfrak{a})$ thus
\[
C_{\mathfrak{g}}(x)\cap Crit(V)=C_{\mathfrak{g}}(\mathfrak{a})\cap
Crit(V)=\mathfrak{a.}%
\]
This implies that $k\mathfrak{a=a}$. Also $k^{-1}hx=x$. Thus the choice of $x$
implies that $k^{-1}h_{|\mathfrak{a}}$ is the identity.
\end{proof}

\subsection{The structure of $X_{h}$ for $h$ generic}

We maintain the notation of the previous subsection and we assume as we may
that $[\mathfrak{a,a}^{\ast}]=0$. The next result uses an argument in
\cite{GoodWall} 12.4.12 in its proof.

We note that since $W_{H}(\mathfrak{a})$ is a subgroup of $GL(\mathfrak{a})$
the set of $x\in\mathfrak{a}$ such that $|W_{H}(\mathfrak{a})x|=|W_{H}%
(\mathfrak{a})|$ is a Zariski open dense subset $\mathfrak{a}^{o}%
\subset\mathfrak{a}$. We note that if $m=2$ then $\mathfrak{a}^{o}%
=\mathfrak{a}^{\prime}$.

\begin{theorem}
If $h\in\mathfrak{a}^{\prime}\cap\mathfrak{a}^{o}$ then $\mathcal{I}_{h}$ is a
radical ideal hence prime.
\end{theorem}

\begin{proof}
The part of the proof of Proposition 12.4.12 in \cite{GoodWall} that shows
that the ideal $\mathcal{I}_{h}$ (in the context of that book) is a radical
ideal that starts on line 10 on p. 603 and continues through line -11 on p.
604 left unchanged in this more general context (except that we must replace
$\mathfrak{a}^{\prime}$ by $\mathfrak{a}^{\prime}\cap\mathfrak{a}^{o}$) proves
the result.
\end{proof}

\begin{lemma}
The set of $h\in\mathfrak{a}^{\prime}\cap\mathfrak{a}^{o}$ such that $X_{h}$
is a smooth affine variety contains a Zariski open dense subset,
$\mathfrak{a}^{\prime\prime}$.
\end{lemma}

\begin{proof}
We have seen that $X_{h}=H(h+\mathfrak{n})$. Let $f_{1},...,f_{r}$ be
algebraically independent homogeneous generators for $\mathcal{O}(V)^{H}$. We
note that if $g\in H,$ $h\in\mathfrak{a}^{\prime}$, $v\in\mathfrak{a,}$
$f\in\mathcal{O}(V)^{H}$ and $x\in\mathfrak{n}$ then
\[
df_{g(h+x)}(gv)=\frac{d}{dt_{t=0}}f(g(h+x+tv))=
\]%
\[
\frac{d}{dt_{t=0}}f(h+tv+x)=\frac{d}{dt_{t=0}}f(h+tv)
\]
since $h+tv$ is semi-simple and $x$ is nilpotent and $[x,h+tv]=0$. Thus if
$z=g(h+x),$ if $u_{j}=f_{j|\mathfrak{a}}$ and if $v_{1},...,v_{s}$ is a basis
of $\mathfrak{a}$ with corresponding linear coordinates $x_{1},...,x_{n}$ then
we we have
\[
(df_{1_{z}}\wedge\cdots\wedge df_{r_{z}})(gv_{1},...,gv_{s})=\det
(\frac{\partial u_{i}}{\partial x_{j}}(h)).
\]
Now $u_{1},...,u_{r}$ are algebraically independent on $\mathfrak{a}$ so the
Jacobian criterion implies that the polynomial $\det(\frac{\partial u_{i}%
}{\partial x_{j}})$ is not identically $0$ on $\mathfrak{a}$. Take
$\mathfrak{a}^{\prime\prime}=\{h\in\mathfrak{a}^{\prime}\cap\mathfrak{a}%
^{o}|\det(\frac{\partial u_{i}}{\partial x_{j}})(h)\neq0\}$. If $h\in
\mathfrak{a}^{\prime\prime}$ then
\[
\dim(T_{z}(X_{h})=\dim V-r
\]
for all $z\in X_{h}$.
\end{proof}

We note that if $m=2$ and if $h\in\mathfrak{a}^{\prime}$ then $X_{h}=Hh$ so
the lemma above is obvious in this case.

We set $M=C_{H}(\mathfrak{a})$ and define for $m\in M$, $g\in H$,
$x\in\mathfrak{n}$, $(g,x)m=(gm,m^{-1}x)$. Then $\left(  H\times
\mathfrak{n}\right)  /M$ is the vector bundle $H\times_{M}\mathfrak{n}$ over
$H/M.$

\begin{theorem}
\label{bundle}Fix $h\in\mathfrak{a}^{\prime\prime}$. If we define $\Psi
_{h}:H\times\mathfrak{n}\rightarrow X_{h}$ by $\Psi_{h}(g,x)=g(h+x)$ then
$\Psi_{h}(g,x)$ depends only on $(g,x)M$ and the induced map of $H\times
_{M}\mathfrak{n}$ to $X_{h}$ is an isomorphism of algebraic varieties.
\end{theorem}

\begin{proof}
Since $h\in\mathfrak{a}^{\prime\prime}$, $X_{h}$ is smooth hence it is a
complex manifold of dimension $n=\dim V-\dim\mathfrak{a}$. We also note that
$H\times_{M}\mathfrak{n}$ is also a smooth variety of the same dimension.
Suppose $\Psi_{h}(g,x)=\Psi_{h}(g^{\prime},x^{\prime})$ with $g,g^{\prime}\in
H$ and $x,x^{\prime}\in\mathfrak{n}$ then $g(h+x)=g^{\prime}(h+x^{\prime})$.
This implies (using The Jordan decomposition) that $gh=g^{\prime}h$. Thus
$g^{-1}g^{\prime}h=h$. So $g^{\prime}=gm$ with $m\in M$. Also $gx=g^{\prime
}x^{\prime}$. Thus $g^{\prime}(h+x^{\prime})=gm(h+m^{-1}x)$. This implies that%
\[
\Psi_{h}:H\times_{M}\mathfrak{n}\rightarrow X_{h}%
\]
is regular and bijective.

We calculate the differential of $\Psi_{h}$ at $g,x$ for $g\in H$ and
$x\in\mathfrak{n}$. Let $X\in\mathrm{Lie}(H)$ and $v\in\mathfrak{n}$. Then%
\[
\left(  d\Psi_{h}\right)  _{g,x}(X,v)=g(X(h+x)+v).
\]
We assert that the dimension of the image of $(d\Psi_{h})_{g,x}$ is $\dim
V-\dim\mathfrak{a}$ for all $g\in H$ and $x\in\mathfrak{n}$. It is clearly
enough to prove this for $g=I$. Let $y\in\mathfrak{g}_{\zeta^{-1}}$ be such
that $B(y,z)=0$ for all $z=X(h+x)+v$ as above. Then
\[
0=B(y,[X,h+x])=B([h+x,y],X)
\]
for all $X\in\mathrm{Lie}(H)$. But $[h+x,y]\in\mathrm{Lie}(H)$ so this implies
that $[h+x,y]=0$. This implies $y\in C_{\mathfrak{g}}(h)_{\zeta^{-1}%
}=C_{\mathfrak{g}}(\mathfrak{a)}_{\zeta^{-1}}$ since $(h+x)_{s}=h$ and
$h\in\mathfrak{a}^{\prime}$. Also $B(y,\mathfrak{n})=0$ implies that
$y\in\mathfrak{a}^{\ast}$. This implies the dimension estimate. We therefore
see that $\left(  d\Psi_{h}\right)  _{g,x}$ is of maximal rank for all $g\in
H,x\in\mathfrak{n}$ so the inverse function theorem implies that
\[
\Psi_{h}:H\times_{M}\mathfrak{n}\rightarrow X_{h}%
\]
is biholomorphic. We assert that $\Psi_{h}$ is also birational. Indeed, if
$x\in\mathfrak{n}$ is such that $H(h+x)$ is open in $X_{h}$ then if $g\in H$
is such that $g(h+x)=h+x$ then the uniqueness of the Jordan decomposition
implies that $gh=h$ and $gx=x.$ Thus $g\in C_{H}(\mathfrak{a})_{x}$. Thus the
open orbit is biregularly isomorphic with $H/C_{H}(\mathfrak{a})_{x}$. We now
consider the same $x$ and the orbit under $H$ of $(e,x)$ in $H\times
_{C_{H}(\mathfrak{a})}\mathfrak{n}$. The stabilizer is the set of $g\in H$
such that $g\in C_{H}(\mathfrak{a})$ and $gx=x$. Thus it is exactly the same.
Also $\dim X_{h}=\dim H\times_{H}\mathfrak{n}$ so the orbit of $(e,x)M$ is
$\mathrm{Zariski}$ open and is regularly isomorphic to the open orbit in
$V_{h}$ under the map $\Psi_{h}$. Thus $\Psi_{h}$ is a birational isomorphism.
The result now follows from the following lemma.
\end{proof}

\begin{lemma}
Let $X$ and $Y$ be smooth irreducible affine varieties of the same dimension
\[
F:X\rightarrow Y
\]
be regular, biholomorphic and birational then $F$ is a regular isomorphism of varieties.
\end{lemma}

\begin{proof}
Let $F^{-1}:Y\rightarrow X$ then $F^{-1}$ is a rational map that is also
holomorphic. We may assume that $X\subset\mathbb{C}^{n}$ as a
$\mathrm{Zariski}$ closed subset. Then $F^{-1}=(\phi_{1},...,\phi_{n})$ with
$\phi_{j},j=1,...,n$ both rational and holomorphic on $Y$. This implies that
if $p\in Y$ then the germ at $p$ of each $\phi_{j}$ is in $\mathcal{O}_{X,p}$
(see the Lemma on p.177 in \cite{Shaf2} which follows from the fact that since
$X$ is smooth $\mathcal{O}_{X,p}$ is a unique factorization domain). Thus each
$\phi_{j}$ is regular and so $F^{-1}$ is regular.
\end{proof}

\subsection{The multiplicity formula}

We consider the representation of $H$ on the harmonics $\mathcal{H}$ (see
Corollary \ref{Harmonics}). Our generalization of the Kostant-Rallis
decomposition of the harmonics is

\begin{theorem}
If $U$ is an irreducible regular $H$--module then%
\[
\dim\mathrm{Hom}_{H}(U,\mathcal{H})=\dim\mathrm{Hom}_{M}(U,\mathcal{O}%
(\mathfrak{n})).
\]

\end{theorem}

We will call the Vinberg pair tame if $\mathfrak{n}=\{0\}$. In particular, if
$\theta^{2}=1$ or the pair is regular ($T_{\mathfrak{a}}$ is a maximal torus
in $G$) then the pair is tame. Thus the theorem in this context is an exact
generalization to the multiplicity theorem of Kostant-Rallis.

\begin{corollary}
If the pair is tame and if $U$ is an irreducible regular $H$--module then%
\[
\dim\mathrm{Hom}_{H}(U,\mathcal{H})=\dim\mathrm{Hom}_{M}(U,\mathbb{C}).
\]

\end{corollary}

Note that the corollary follows directly from the above theorem and Theorem
12.4.13 in \cite{GoodWall} .

We will devote the rest of this subsection to the prove of this theorem. First
we need

\begin{lemma}
The $H$--module $\mathcal{H}$ is equivalent to $\mathcal{O}(V)/\mathcal{I}%
_{h}$ for any $h\in\mathfrak{a}$.
\end{lemma}

\begin{proof}
We put the natural filtration by degree on $\mathcal{O}(V)/\mathcal{I}_{h}$.
Then Lemma 12.4.9 in \cite{GoodWall} immediately implies that $Gr(\mathcal{O}%
(V)/\mathcal{I}_{h})$ is isomorphic with $\mathcal{O}(V)/\mathcal{I}%
_{0}=\mathcal{O}(V)/\mathcal{O}(V)\mathcal{O}_{+}(V)^{H}$. which we have seen
is isomorphic with $\mathcal{H}$ as an $H$ module.
\end{proof}

We recall that if $h\in\mathfrak{a}^{\prime\prime}\mathcal{\ }$then
$\mathcal{I}_{h}$ is prime. So the lemma above implies that $\mathcal{H}$ is
isomorphic with $Gr(\mathcal{O}(X_{h}))$ as a representation of $H$ for any
$h\in\mathfrak{a}^{\prime\prime}$ (see the previous section). The theorem now
follows from

\begin{proposition}
Let $h\in\mathfrak{a}^{\prime\prime}$. If $U$ is an irreducible regular $H
$--module then%
\[
\dim\mathrm{Hom}_{H}(U,\mathcal{O}(X_{h}))=\dim\mathrm{Hom}_{M}(U,\mathcal{O}%
(\mathfrak{n})).
\]

\end{proposition}

\begin{proof}
In light of Theorem \ref{bundle} may replace $\mathcal{O}(X_{h})$ for
$h\in\mathfrak{a}^{\prime\prime}$ with $\mathcal{O}(H\times_{M}\mathfrak{n})$.
By our definition $H\times_{M}\mathfrak{n}=(H\times\mathfrak{n})/M$ using the
right action above, $\mathcal{O}((H\times\mathfrak{n)}/M)=\mathcal{O}%
(H\times\mathfrak{n})^{M}$. Here $M$ acts on $\mathcal{O}(H\times
\mathfrak{n})$ by $mf(g,x)=f(gm,m^{-1}x)$ for $g\in H,x\in\mathfrak{n}$ and
$m\in M$. Now%
\[
\mathcal{O}(H\times\mathfrak{n})\cong\mathcal{O}(H)\otimes\mathcal{O}%
(\mathfrak{n})
\]
(under the map $\left(  u\otimes v\right)  (g,x)=u(g)v(x)$, $u\in
\mathcal{O}(H),v\in\mathcal{O}(\mathfrak{n})$) and the action of $M$ is just
the tensor product action relative to the right action on $H$ and the left
action on $\mathfrak{n}$. Thus $M$ leaves invariant the grade on
$\mathcal{O}(\mathfrak{n})$. So
\[
\mathcal{O}(H\times\mathfrak{n})^{M}=\bigoplus_{j\geq0}\left(  \mathcal{O}%
(H)\otimes\mathcal{O}^{j}(\mathfrak{n})\right)  ^{M}.
\]
If $f\in\left(  \mathcal{O}(H)\otimes\mathcal{O}^{j}(\mathfrak{n})\right)
^{M}$ (as a subspace of $\left(  \mathcal{O}(H)\otimes\mathcal{O}%
(\mathfrak{n})\right)  ^{M}$) then we define $\mathbf{f}(g)(x)=f(g,x)$ for
$g\in H$ and $x\in\mathfrak{n}.$Then
\[
\mathbf{f}:H\rightarrow\mathcal{O}^{j}(\mathfrak{n})
\]
is regular and $\mathbf{f}(gm)=m^{-1}\mathbf{f}(g)$. That is, as an
$H$--module,
\[
\left(  \mathcal{O}(H)\otimes\mathcal{O}^{j}(\mathfrak{n})\right)  ^{M}%
\cong\mathrm{Ind}_{M}^{H}(\mathcal{O}^{j}(\mathfrak{n})).
\]
The theorem now follows from Frobenius reciprocity. (See e.g. \cite{GoodWall}
section 12.1.2 for the undefined terms and the reciprocity.)
\end{proof}

\section{Examples for E$_{6}$ and E$_{8}$}

The full details of this discussion can be found in \cite{Wallach} also most
of the preliminaries to the actual multiplicity formula can be found in
\cite{Elesh-Vin}. However the interested reader can take the unproved
assertions in this paper to be exercises.

\subsection{An E$_{6}$ example}

We take $\mathfrak{g}$ to be simple of type $E_{6}$. Fix a Cartan subalgebra
$\mathfrak{h}$ and a system of positive roots. The simple roots are
$\alpha_{1},...,\alpha_{6}$ and the extended Dynkin diagram in the Bourbaki
ordering is
\[%
\begin{array}
[c]{ccccccccc}
&  &  &  & \circ & -\beta &  &  & \\
&  &  &  & | &  &  &  & \\
&  &  &  & \circ & \alpha_{2} &  &  & \\
&  &  &  & | &  &  &  & \\
\circ & \frac{\qquad}{{}} & \circ & \frac{\qquad}{{}} & \circ & \frac{\qquad
}{{}} & \circ & \frac{\qquad}{{}} & \circ\\
\alpha_{1} &  & \alpha_{3} &  & \alpha_{4} &  & \alpha_{5} &  & \alpha_{6}%
\end{array}
\]
Let $H_{1},...,H_{6}$ be the dual basis of $\mathfrak{h}$ to the simple roots
(i.e. $\alpha_{i}(H_{j})=\delta_{ij}$). Then the automorphism $\mathfrak{g}$
given by $\theta=\exp(\frac{2\pi i}{3}\mathrm{ad}H_{4})$ is of order $3$ since
the coefficient of $\alpha_{4}$ in the expansion of the highest root, $\beta$,
is $3$. In this case we see that we have $H$ is locally isomorphic with
$SL(3,\mathbb{C})\times SL(3,\mathbb{C})\times SL(3,\mathbb{C})$ (since its
Dynkin diagram is gotten by deleting the node labeled $\alpha_{4}$) and
$\mathfrak{g}$ is the direct sum of $\mathrm{Lie}(H)$ and a direct sum of
$H$--modules
\[
\mathbb{C}^{3}\otimes\mathbb{C}^{3}\otimes\mathbb{C}^{3}\oplus\left(
\mathbb{C}^{3}\otimes\mathbb{C}^{3}\otimes\mathbb{C}^{3}\right)  ^{\ast}%
\]
The corresponding Vinberg pair is
\[
(SL(3,\mathbb{C})\otimes SL(3,\mathbb{C})\otimes SL(3,\mathbb{C}%
),\mathbb{C}^{3}\otimes\mathbb{C}^{3}\otimes\mathbb{C}^{3}).
\]
(Here the indicated group is the set of elements $g_{1}\otimes g_{2}\otimes
g_{3}$ with $g_{i}\in SL(3,\mathbb{C})$.) Let $e_{1},e_{2},e_{3}$ denote the
standard basis of $\mathbb{C}^{3}$. One can show that%
\[
v_{1}=e_{1}\otimes e_{1}\otimes e_{1}+e_{2}\otimes e_{2}\otimes e_{2}%
+e_{3}\otimes e_{3}\otimes e_{3},
\]%
\[
v_{2}=e_{1}\otimes e_{2}\otimes e_{3}+e_{3}\otimes e_{1}\otimes e_{2}%
+e_{2}\otimes e_{3}\otimes e_{1},
\]%
\[
v_{3}=e_{3}\otimes e_{2}\otimes e_{1}+e_{1}\otimes e_{3}\otimes e_{2}%
+e_{2}\otimes e_{1}\otimes e_{3}%
\]
then $v_{1},v_{2},v_{3}$ is a basis of a Cartan subspace, $\mathfrak{a}$, of
$V$. To see this we note that if $T$ is the product of the diagonal Cartan
subgroups then the weights of $V$ are all of multiplicity one. This implies
that $\wedge^{2}V$ is multiplicity free and since $V^{\ast}$occurs in
$\wedge^{2}V$ we see that the bracket of $\mathfrak{g}$ restricted to $V$ is
up to scalar multiple given by
\[
\lbrack x_{1}\otimes x_{2}\otimes x_{3},y_{1}\otimes y_{2}\otimes y_{3}%
]=x_{1}\wedge y_{1}\otimes x_{2}\wedge y_{2}\otimes x_{3}\wedge y_{3}%
\]
with $\wedge^{2}\mathbb{C}^{3}$ identified with $\left(  \mathbb{C}%
^{3}\right)  ^{\ast}$. Observe that this implies that $[v_{i},v_{j}]=0$ all
$i,j$. Also a direct calculation shows that
\[
\left\langle E_{rs}v_{i},v_{j}\right\rangle =0
\]
and
\[
\left\langle \left(  E_{rr}-E_{ss}\right)  v_{i},v_{j}\right\rangle =0
\]
if $r\neq s$ for all $i,j$. Thus the span of the $v_{i}$ is abelian and
consists of critical. hence semi--simple, elements. Finally since
$\varphi(3)=2$ and rank $E_{6}$ is $6$ we see that since a Cartan subspace is
at most of dimension $\frac{\mathrm{rank}(\mathfrak{g)}}{\varphi(m)}=3$ this
span, indeed a Cartan subspace. We also note that

1. $C_{\mathfrak{g}}(\mathfrak{a})\cap V=\mathfrak{a}$. \smallskip

2. The group $M$ is the set of triples of matrices $M_{1}\cup M_{2}\cup
M_{3}\cup M_{4}$ with $\alpha,\beta,\delta,\mu$ third roots of $1.$
\[
M_{1}=\left\{  \left(  \alpha\left[
\begin{array}
[c]{ccc}%
1 & 0 & 0\\
0 & 1 & 0\\
0 & 0 & 1
\end{array}
\right]  ,\beta\left[
\begin{array}
[c]{ccc}%
1 & 0 & 0\\
0 & 1 & 0\\
0 & 0 & 1
\end{array}
\right]  ,\frac{1}{\alpha\beta}\left[
\begin{array}
[c]{ccc}%
1 & 0 & 0\\
0 & 1 & 0\\
0 & 0 & 1
\end{array}
\right]  \right)  \right\}  ,
\]%
\[
M_{2}=\left\{  \left(  \alpha\left[
\begin{array}
[c]{ccc}%
\delta & 0 & 0\\
0 & \mu & 0\\
0 & 0 & \frac{1}{\delta\mu}%
\end{array}
\right]  ,\beta\left[
\begin{array}
[c]{ccc}%
\delta & 0 & 0\\
0 & \mu & 0\\
0 & 0 & \frac{1}{\delta\mu}%
\end{array}
\right]  ,\frac{1}{\alpha\beta}\left[
\begin{array}
[c]{ccc}%
\delta & 0 & 0\\
0 & \mu & 0\\
0 & 0 & \frac{1}{\delta\mu}%
\end{array}
\right]  \right)  |\delta\neq\mu\right\}  ,
\]%
\[
M_{3}=\left\{  \left(  \alpha\left[
\begin{array}
[c]{ccc}%
0 & \delta & 0\\
0 & 0 & \mu\\
\frac{1}{\delta\mu} & 0 & 0
\end{array}
\right]  ,\beta\left[
\begin{array}
[c]{ccc}%
0 & \delta & 0\\
0 & 0 & \mu\\
\frac{1}{\delta\mu} & 0 & 0
\end{array}
\right]  ,\frac{1}{\alpha\beta}\left[
\begin{array}
[c]{ccc}%
0 & \delta & 0\\
0 & 0 & \mu\\
\frac{1}{\delta\mu} & 0 & 0
\end{array}
\right]  \right)  \right\}  ,
\]%
\[
M_{4}=\left\{  \left(  \alpha\left[
\begin{array}
[c]{ccc}%
0 & 0 & \delta\\
\mu & 0 & 0\\
0 & \frac{1}{\delta\mu} & 0
\end{array}
\right]  ,\beta\left[
\begin{array}
[c]{ccc}%
0 & 0 & \delta\\
\mu & 0 & 0\\
0 & \frac{1}{\delta\mu} & 0
\end{array}
\right]  ,\frac{1}{\alpha\beta}\left[
\begin{array}
[c]{ccc}%
0 & 0 & \delta\\
\mu & 0 & 0\\
0 & \frac{1}{\delta\mu} & 0
\end{array}
\right]  \right)  \right\}  .
\]

3. The order of $M$ is $81$

4. Every element of $M_{i}$ for $i>1$ is conjugate to
\[
\left(  \left[
\begin{array}
[c]{ccc}%
\zeta^{2} & 0 & 0\\
0 & \zeta & 0\\
0 & 0 & 1
\end{array}
\right]  ,\left[
\begin{array}
[c]{ccc}%
\zeta^{2} & 0 & 0\\
0 & \zeta & 0\\
0 & 0 & 1
\end{array}
\right]  ,\left[
\begin{array}
[c]{ccc}%
\zeta^{2} & 0 & 0\\
0 & \zeta & 0\\
0 & 0 & 1
\end{array}
\right]  \right)
\]
in $H$ with $\zeta=e^{\frac{2\pi i}{3}}$.\smallskip

We parametrize the irreducible regular representations of $SL(3,\mathbb{C})$
by pairs of integers $m\geq n\geq0$ as the restrictions of the irreducible
representation of $GL(3,\mathbb{C})$ corresponding to $m\geq n\geq0$ (c.f.
\cite{GoodWall}Theorem 5.5.22). This parametrization is by the highest weight
$m\varepsilon_{1}+n\varepsilon_{2}$ restricted to the diagonal matrices of
trace $0$. We write the representation as $F^{m,n}$. Thus the irreducible
regular representations of $H$ are of the form $F^{m_{1},n_{1}}\otimes
F^{m_{2},n_{2}}\otimes F^{m_{3},n_{3}}$.\smallskip

We have

6. $F^{m_{1},n_{1}}\otimes F^{m_{2},n_{2}}\otimes F^{m_{3},n_{3}}$ has a fixed
vector for the group $M_{1}$ above if and only if $m_{1}+n_{1}\equiv
m_{2}+n_{2}\equiv n_{3}+m_{3}\operatorname{mod}3$.

\begin{proposition}
If the condition of Exercise 4 is not satisfied then $\mathrm{Hom}%
_{H}(F^{m_{1},n_{1}}\otimes F^{m_{2},n_{2}}\otimes F^{m_{3},n_{3}}%
,\mathcal{H})=\{0\}$. If it is satisfied then
\[
\dim F^{m_{1},n_{1}}\otimes F^{m_{2},n_{2}}\otimes F^{m_{3},n_{3}%
}\operatorname{mod}9\in\{0,1,8\}.
\]
Set $\varepsilon(F^{m_{1},n_{1}}\otimes F^{m_{2},n_{2}}\otimes F^{m_{3},n_{3}%
})=0,8,-8$ respectively if the congruence modulo $9$ is $0,1$ or $8$. Then
\[
\mathrm{Hom}_{H}(F^{m_{1},n_{1}}\otimes F^{m_{2},n_{2}}\otimes F^{m_{3},n_{3}%
},\mathcal{H})=
\]%
\[
\frac{\dim F^{m_{1},n_{1}}\otimes F^{m_{2},n_{2}}\otimes F^{m_{3},n_{3}%
}+\varepsilon(F^{m_{1},n_{1}}\otimes F^{m_{2},n_{2}}\otimes F^{m_{3},n_{3}}%
)}{9}
\]

\end{proposition}

We will prove the proposition using\smallskip

7. \textbf{\ }Let $G$ be a group and $X$ a finite dimensional $G$--module with
character $\chi_{X}$. If $M$ is a finite subgroup of $G$ then%
\[
\dim V^{M}=\frac{1}{|M|}\sum_{m\in M}\chi_{X}(m).
\]

We can now prove the result. The order of $M$ is $81.$ If $X=F^{m_{1},n_{1}%
}\otimes F^{m_{2},n_{2}}\otimes F^{m_{3},n_{3}}$ and it satisfies the
congruence condition in 6. then the value of $\chi_{X}$ on each element of
$M_{1}$ is $\dim X$. There are $9$ such elements. Set $\chi_{m,n}%
=\chi_{F^{m,n}}$. 4. implies that the other $72$ elements of $M$ all have the
value
\[
\chi_{m_{1},n_{1}}(u)\chi_{m_{2},n_{2}}(u)\chi_{m_{3},n_{3}}(u)
\]
with
\[
u=\left[
\begin{array}
[c]{ccc}%
\zeta^{2} & 0 & 0\\
0 & \zeta & 0\\
0 & 0 & 1
\end{array}
\right]  .
\]
Using the Weyl character formula , the Weyl denominator formula and 7. above
the proposition follows as an exercise here is a hint

A similar result in the more complicated context of $E_{8}$ is proved in the
next subsection with some of the same ideas. Let $T$ be the diagonal torus in
$SL(3,\mathbb{C})$ and let
\[
\varepsilon_{i}\left[
\begin{array}
[c]{ccc}%
x_{1} & 0 & 0\\
0 & x_{2} & 0\\
0 & 0 & x_{3}%
\end{array}
\right]  =x_{i},i=1,2,3.
\]
Then $\rho$ the half sum of the positive roots is $\varepsilon_{1}%
-\varepsilon_{3}$. Then
\[
H_{\rho}=\left[
\begin{array}
[c]{ccc}%
1 & 0 & 0\\
0 & 0 & 0\\
0 & 0 & -1
\end{array}
\right]
\]
so $u=e^{\frac{2\pi i}{3}H_{\rho}}$. The Weyl denominator formula implies
that
\[
\sum_{s\in S_{3}}sgn(s)e^{s\rho(H)}=e^{\rho(H)}(1-e^{-\left(  \varepsilon
_{1}-\varepsilon_{2}\right)  (H)})(1-e^{-\left(  \varepsilon_{2}%
-\varepsilon_{3}\right)  (H)})(1-e^{-\left(  \varepsilon_{1}-\varepsilon
_{3}\right)  (H)})
\]
The Weyl character formula says that if $\Lambda=m\varepsilon_{1}%
+n\varepsilon_{2}$ then%
\[
\chi_{m,n}(e^{H})=\frac{\sum_{s\in S_{3}}sgn(s)e^{s(\Lambda+\rho)(H)}}%
{\sum_{s\in S_{3}}sgn(s)e^{s\rho(H)}}.
\]
So%
\[
\chi_{m,n}(u)=\chi_{m,n}(e^{H_{\rho}})=\frac{\sum_{s\in S_{3}}%
sgn(s)e^{s(\Lambda+\rho)(H_{\rho})}}{\sum_{s\in S_{3}}sgn(s)e^{s\rho(H_{\rho
})}}=
\]%
\[
\frac{\sum_{s\in S_{3}}sgn(s)e^{s\rho(H_{\Lambda+\rho})}}{\sum_{s\in S_{3}%
}sgn(s)e^{s\rho(H_{\rho})}}%
\]
with
\[
H_{\Lambda+\rho}=\left[
\begin{array}
[c]{ccc}%
m+1 & 0 & 0\\
0 & n & 0\\
0 & 0 & -1
\end{array}
\right]
\]
now apply the denominator formula and calculate.

\subsection{An E$_{8}$ example}

We take $\mathfrak{g}$ to be simple of type $E_{8}$. Fix a Cartan subalgebra
$\mathfrak{h}$ and a system of positive roots. The simple roots are
$\alpha_{1},...,\alpha_{8}$ and the extended Dynkin diagram in the Bourbaki
ordering is%
\[%
\begin{array}
[c]{ccccccccccccccc}
&  &  &  & \circ & \alpha_{2} &  &  &  &  &  &  &  &  & \\
&  &  &  & | &  &  &  &  &  &  &  &  &  & \\
\circ & \frac{\quad}{{}} & \circ & \frac{\quad}{{}} & \circ & \frac{\quad}{{}}
& \circ & \frac{\quad}{{}} & \circ & \frac{\quad}{{}} & \circ & \frac{\quad
}{{}} & \circ & \frac{\quad}{{}} & \circ\\
\alpha_{1} &  & \alpha_{3} &  & \alpha_{4} &  & \alpha_{5} &  & \alpha_{6} &
& \alpha_{7} &  & \alpha_{8} &  & -\beta
\end{array}
.
\]

As before we take the dual basis to the simple roots $\alpha_{i}(H_{j}%
)=\delta_{ij}$. In this case the coefficient of $\alpha_{2}$ in $\beta$ is $3$
so $\theta=\exp(\frac{2\pi i}{3}\mathrm{ad}H_{2})$ is and automorphism of
$\mathfrak{g}$ of order $3$. This yields the Vinberg pair $(H,V)=(SL(9,\mathbb{C}),\wedge^{3}\mathbb{C}^{9})$. 
For simplicity we will use the simply connected covering group $SL(9,\mathbb{C})$ then we note that the
covering map $\tilde{H}=SL(9,\mathbb{C})\rightarrow\wedge^{3}SL(9,\mathbb{C})$
has kernel $S=\{zI|z^{3}=1\}$. We also note that a Cartan subspace in $V$ is
the space $\mathfrak{a}$ with basis%
\[
\omega_{1}=e_{1}\wedge e_{2}\wedge e_{3}+e_{4}\wedge e_{5}\wedge e_{6}%
+e_{7}\wedge e_{8}\wedge e_{9},
\]%
\[
\omega_{2}=e_{1}\wedge e_{4}\wedge e_{7}+e_{2}\wedge e_{5}\wedge e_{8}%
+e_{3}\wedge e_{6}\wedge e_{9},
\]%
\[
\omega_{3}=e_{1}\wedge e_{5}\wedge e_{9}+e_{2}\wedge e_{6}\wedge e_{7}%
+e_{3}\wedge e_{4}\wedge e_{8},
\]%
\[
\omega_{4}=e_{1}\wedge e_{6}\wedge e_{8}+e_{2}\wedge e_{4}\wedge e_{9}%
+e_{3}\wedge e_{5}\wedge e_{7}.
\]
Indeed one checks that $\left\langle X\omega_{i},\omega_{j}\right\rangle =0 $
for $X=E_{ij},i<j$ and $X=E_{ii}-E_{i+1,i+1},i=1,...,8$. Thus every element of
the span of $\omega_{1},\omega_{2},\omega_{3},\omega_{4}$is critical and so
the Kempf-Ness theorem implies that $\tilde{H}v$ is closed for every element
in $\mathfrak{a}$. Up to scalar multiple the bracket in E$_{8}$ of $v,w\in V$
is given by $v\wedge w$. since the weights in $\wedge^{3}\mathbb{C}^{9}$ are
multiplicity at most 1 (in general the multiplicity of an extreme weight $\xi$
in a tensor product of $F^{\Lambda}\otimes F^{\mu}$ is at most the
multiplicity of the weight $\xi-\Lambda$ in $F^{\mu}$. This also implies that
$[\omega_{i},\omega_{j}]=0$ for all $i,j$.\bigskip

The centralizer of $\mathfrak{a}$ in $H$, $C=C_{H}(\mathfrak{a})$, is the
intersection of $H$ with $T_{\mathfrak{a}}$. Thus $C$ is abelian. We thus have
the exact sequence%
\[
1\rightarrow S\rightarrow C_{\tilde{H}}(\mathfrak{a})\rightarrow
C\rightarrow1.
\]
The following elements are obviously in $C_{\tilde{H}}(\mathfrak{a})$ (Here
$w$and $z$ are third roots of $1$):%
\[
A_{z,w}=w\left[
\begin{array}
[c]{ccc}%
I & 0 & 0\\
0 & zI & 0\\
0 & 0 & z^{2}I
\end{array}
\right]  ,
\]%
\[
B_{z,w}=w\left[
\begin{array}
[c]{ccc}%
\left[
\begin{array}
[c]{ccc}%
1 &  & \\
& z & \\
&  & z^{2}%
\end{array}
\right]  & 0 & 0\\
0 & \left[
\begin{array}
[c]{ccc}%
z^{2} &  & \\
& 1 & \\
&  & z
\end{array}
\right]  & 0\\
0 & 0 & \left[
\begin{array}
[c]{ccc}%
z &  & \\
& z^{2} & \\
&  & 1
\end{array}
\right]
\end{array}
\right]  .
\]
Thus if $g\in\tilde{C}$ and%
\[
g=\left[
\begin{array}
[c]{ccc}%
X_{1} & X_{2} & X_{3}\\
Y_{1} & Y_{2} & Y_{3}\\
Z_{1} & Z_{2} & Z_{3}%
\end{array}
\right]
\]
If $z$ is a primitive third root of $1$ then%
\[
A_{z,1}gA_{z^{2\noindent},1}=wg
\]
with $w=1,z$, or $z^{2}$. We have the following three cases.

\noindent a) $w=1$: $g$ is block diagonal $\left[
\begin{array}
[c]{ccc}%
X_{1} & 0 & 0\\
0 & Y_{2} & 0\\
0 & 0 & Z_{3}%
\end{array}
\right]  .$

\noindent b) $w=z$: $g$ has block form $\left[
\begin{array}
[c]{ccc}%
0 & 0 & X_{3}\\
Y_{1} & 0 & 0\\
0 & Z_{2} & 0
\end{array}
\right]  .$

\noindent c) $w=z^{2}$: $g$ has block form $\left[
\begin{array}
[c]{ccc}%
0 & X_{2} & 0\\
0 & 0 & Y_{3}\\
Z_{1} & 0 & 0
\end{array}
\right]  .$

We now observe the relationship between case a) and the previous example for
E$_{6}$. Let%
\[
V_{1}=\mathbb{C}e_{1}\oplus\mathbb{C}e_{4}\oplus\mathbb{C}e_{7},
\]%
\[
V_{2}=\mathbb{C}e_{2}\oplus\mathbb{C}e_{5}\oplus\mathbb{C}e_{8},
\]%
\[
V_{3}=\mathbb{C}e_{3}\oplus\mathbb{C}e_{6}\oplus\mathbb{C}e_{9}.
\]
We have a linear isomorphism $T:\mathbb{C}^{3}\otimes\mathbb{C}^{3}%
\otimes\mathbb{C}^{3}\rightarrow\wedge^{3}\mathbb{C}^{9}$ given by%
\[
e_{i}\otimes e_{j}\otimes e_{k}\mapsto e_{i}\wedge e_{j+3}\wedge e_{k+6},1\leq
i,j,k\leq3.
\]
Under this map we have the intertwining%
\[
T\circ g_{1}\otimes g_{2}\otimes g_{3}=%
{\displaystyle\bigwedge\nolimits^{3}}
\left[
\begin{array}
[c]{ccc}%
g_{1} & 0 & 0\\
0 & g_{2} & 0\\
0 & 0 & g_{3}%
\end{array}
\right]  ,g_{i}\in SL(3,\mathbb{C}),i=1,2,3.
\]
We also note that that
\[
T^{-1}(\omega_{2})=v_{1}=e_{1}\otimes e_{1}\otimes e_{1}+e_{2}\otimes
e_{2}\otimes e_{2}+e_{3}\otimes e_{3}\otimes e_{3},
\]%
\[
T^{-1}(\omega_{3})=v_{2}=e_{1}\otimes e_{2}\otimes e_{3}+e_{3}\otimes
e_{1}\otimes e_{2}+e_{2}\otimes e_{3}\otimes e_{1},
\]%
\[
T^{-1}(\omega_{4})=v_{3}=e_{3}\otimes e_{2}\otimes e_{1}+e_{1}\otimes
e_{3}\otimes e_{2}+e_{2}\otimes e_{1}\otimes e_{3}.
\]

We will think of $(X,Y,Z)$ as the corresponding block diagonal matrix. The
results of the previous subsection imply we we will find all elements of the
form in case $a)$ if we find the elements in the sets $M_{i}$ that fix
$\omega_{1}$. Here is the list as they come from the $M_{i}$:%
\[
\left\{  \left(  w\left[
\begin{array}
[c]{ccc}%
1 & 0 & 0\\
0 & 1 & 0\\
0 & 0 & 1
\end{array}
\right]  ,w\left[
\begin{array}
[c]{ccc}%
1 & 0 & 0\\
0 & 1 & 0\\
0 & 0 & 1
\end{array}
\right]  ,w\left[
\begin{array}
[c]{ccc}%
1 & 0 & 0\\
0 & 1 & 0\\
0 & 0 & 1
\end{array}
\right]  \right)  |w^{3}=1\right\}  ,
\]%
\[
\left\{  \left(  w\left[
\begin{array}
[c]{ccc}%
1 & 0 & 0\\
0 & z & 0\\
0 & 0 & z^{2}%
\end{array}
\right]  ,w\left[
\begin{array}
[c]{ccc}%
z^{2} & 0 & 0\\
0 & 1 & 0\\
0 & 0 & z
\end{array}
\right]  ,w\left[
\begin{array}
[c]{ccc}%
z & 0 & 0\\
0 & z^{2} & 0\\
0 & 0 & 1
\end{array}
\right]  \right)  |z^{3},w^{3}=1,z\neq1\right\}  ,
\]%
\[
\left\{  \left(  w\left[
\begin{array}
[c]{ccc}%
0 & 1 & 0\\
0 & 0 & z\\
z^{2} & 0 & 0
\end{array}
\right]  ,w\left[
\begin{array}
[c]{ccc}%
0 & z^{2} & 0\\
0 & 0 & 1\\
z & 0 & 0
\end{array}
\right]  ,w\left[
\begin{array}
[c]{ccc}%
0 & z & 0\\
0 & 0 & z^{2}\\
1 & 0 & 0
\end{array}
\right]  \right)  |z^{3},w^{3}=1\right\}  ,
\]%
\[
\left\{  \left(  w\left[
\begin{array}
[c]{ccc}%
0 & 0 & 1\\
z & 0 & 0\\
0 & z^{2} & 0
\end{array}
\right]  ,w\left[
\begin{array}
[c]{ccc}%
0 & 0 & z^{2}\\
1 & 0 & 0\\
0 & z & 0
\end{array}
\right]  ,w\left[
\begin{array}
[c]{ccc}%
0 & 0 & z\\
z^{2} & 0 & 0\\
0 & 1 & 0
\end{array}
\right]  \right)  |w^{3},z^{3}=1\right\}  .
\]
Thus in case a) there are $27$ elements.

We now observe that the elements%
\[
U=\left[
\begin{array}
[c]{ccc}%
0 & I & 0\\
0 & 0 & I\\
I & 0 & 0
\end{array}
\right]  ,V=\left[
\begin{array}
[c]{ccc}%
0 & 0 & I\\
I & 0 & 0\\
0 & I & 0
\end{array}
\right]
\]
are in $C_{\tilde{H}}(\mathfrak{a})$ and the product of $V$ with the elements
in case b) are in the case a) as are the products of $U$ with the elements in
case c) are in the case a). Thus we have a group of order $81$.

Noting that $\,H_{\rho}$ is the diagonal matrix $\mathrm{diag}%
(4,3,2,1,0,-1,-2,-3,-4)$ we have

\begin{lemma}
Any element in $C_{\tilde{H}}(\mathfrak{a})$ that is not a multiple of the
identity is conjugate to
\[
\mu=e^{\frac{2\pi i}{3}H_{\rho}}=\mathrm{diag}(\zeta,1,\zeta^{2},\zeta
,1,\zeta^{2},\zeta,1,\zeta^{2})
\]
with $\zeta=e^{\frac{2\pi i}{3}}$.
\end{lemma}

We will label the irreducible representations of $\tilde{H}=SL(9,\mathbb{C})$
by their highest weight $\Lambda=(\lambda_{1},...,\lambda_{8},0)$ restricted
to the diagonal matrices of trace $0$. Thus a necessary condition for
$F^{\Lambda}$ to occur in $\mathcal{O(\wedge}^{3}\mathbb{C}^{9})$ is that
$\sum_{i=1}^{8}\lambda_{i}\equiv0\operatorname{mod}3$. Let $\chi_{\Lambda}$
denote the character of $F^{\Lambda}$.

\begin{lemma}
If $\sum_{i=1}^{8}\lambda_{i}\equiv0\operatorname{mod}3$ then denoting by
$\mathcal{H}$ the $\tilde{H}$ harmonics in $\mathcal{O(\wedge}^{3}%
\mathbb{C}^{9})$ we have
\[
\dim Hom_{SL(9,\mathbb{C)}}(F^{\Lambda},\mathcal{H})=\frac{\dim F^{\Lambda
}+26\chi_{\Lambda}(\mu)}{27}.
\]

\end{lemma}

\begin{proof}
Frobenius reciprocity and 7. in the previous subsection imply
\[
\dim Hom_{SL(9,\mathbb{C)}}(F^{\Lambda},\mathcal{H})=\frac{1}{\left\vert
C_{\tilde{H}}(\mathfrak{a})\right\vert }\sum_{c\in C_{\tilde{H}}%
(\mathfrak{a})}\chi_{\Lambda}(c).
\]
The above results imply that this expression is equal to%
\[
\frac{3\chi_{\Lambda}(I)+78\chi_{\Lambda}(\mu)}{81}.
\]

\end{proof}

We will now use a variant of Weyl's method of deriving his dimension formula
to calculate $\chi_{\Lambda}(\mu)$. We first consider%
\[
\chi_{\Lambda}(e^{(\frac{2\pi i}{3}+t)H_{\rho}})=\frac{\sum_{s\in S_{9}%
}sgn(s)e^{s(\Lambda+\rho)((\frac{2\pi i}{3}+t)H_{\rho})}}{\sum_{s\in S_{9}%
}sgn(s)e^{s\rho((\frac{2\pi i}{3}+t)H_{\rho})}}.
\]
We want to apply Weyl's denominator formula (using the usual positive roots of
the diagonal Cartan subgroup that is $\varepsilon_{i}-\varepsilon_{j}$ with
$i<j$) to both the numerator and the denominator. Since%
\[
\sum_{s\in S_{9}}sgn(s)e^{s\rho(h)}=e^{\rho(h)}\prod_{\alpha>0}(1-e^{-\alpha
(h)}),
\]
We have%
\[
\chi_{\Lambda}(e^{(\frac{2\pi i}{3}+t)H_{\rho}})=\frac{e^{(\frac{2\pi i}%
{3}+t)\left\langle \Lambda+\rho,\rho\right\rangle }}{e^{(\frac{2\pi i}%
{3}+t)\left\langle \rho,\rho\right\rangle }}%
{\displaystyle\prod\limits_{\alpha>0}}
\frac{(1-e^{-(\frac{2\pi i}{3}+t)\left\langle \alpha,\Lambda+\rho\right\rangle
})}{(1-e^{-(\frac{2\pi i}{3}+t)\left\langle \alpha,\rho\right\rangle })}.
\]
Thus the value we want is gotten by taking the limit as $t\rightarrow0$. In
the denominator the factors that go to $0$ are exactly the ones such that
\[
\alpha(H_{\rho})\equiv0\operatorname{mod}3\text{.}
\]
There are $9$ of these roots which correspond to $\varepsilon_{i}%
-\varepsilon_{j}$ with $j-i=3$ or $6$, with $6$ roots for the value $3$ and
$3$ for the value $6$. Thus to take the limit we must have at least $9$
positive roots with%
\[
\left\langle \alpha,\Lambda+\rho\right\rangle \equiv0\operatorname{mod}%
3\text{.}
\]
If there are more than $9$ then the limit is $0$ and thus in this case
\[
\dim Hom_{SL(9,\mathbb{C)}}(F^{\Lambda},\mathcal{H})=\frac{\dim F^{\Lambda}%
}{27}\text{.}
\]
So suppose that there are exactly $9$. Let $S_{j}(\Lambda)=\{\alpha
|\alpha>0,\left\langle \alpha,\Lambda+\rho\right\rangle \equiv
j\operatorname{mod}3\},j=0,1,2$. Then $S_{0}(\Lambda)=9=S_{0}(0)$ and thus
$S_{1}(\Lambda)+S_{2}(\Lambda)=S_{1}(0)+S_{2}(0)=27$. With this notation
$\chi_{\Lambda}(e^{(\frac{2\pi i}{3}+t)H_{\rho}})$ is given by%
\[
e^{(\frac{2\pi i}{3}+t)\left\langle \Lambda,\rho\right\rangle }\frac
{\prod_{\alpha\in S_{0}(\Lambda)}(1-e^{-t\left\langle \alpha,\Lambda
+\rho\right\rangle })}{\prod_{\alpha\in S_{0}(0)}(1-e^{-t\left\langle
\alpha,\rho\right\rangle })}\frac{\prod_{\alpha\in S_{1}(\Lambda)}(1-\zeta
^{2}e^{-t\left\langle \alpha,\Lambda+\rho\right\rangle })}{\prod_{\alpha\in
S_{1}(0)}(1-\zeta^{2}e^{-t\left\langle \alpha,\rho\right\rangle })}\times
\]%
\[
\frac{\prod_{\alpha\in S_{2}(\Lambda)}(1-\zeta e^{-t\left\langle
\alpha,\Lambda+\rho\right\rangle })}{\prod_{\alpha\in S_{2}(0)}(1-\zeta
e^{-t\left\langle \alpha,\rho\right\rangle })}
\]
Note that $|S_{1}(0)|=15$ and $|S_{2}(0)|=12$. Thus the limit as
$t\rightarrow0$ is%
\[
e^{\frac{2\pi i}{3}\left\langle \Lambda,\rho\right\rangle }\frac{\prod
_{\alpha\in S_{0}(\Lambda)}\left\langle \alpha,\Lambda+\rho\right\rangle
}{\prod_{\alpha\in S_{0}(0)}\left\langle \alpha,\rho\right\rangle }%
\frac{(1-\zeta^{2})^{|S_{1}(\Lambda)|}}{(1-\zeta^{2})^{15|}(1-\zeta)^{12}}=
\]%
\[
e^{\frac{2\pi i}{3}\left\langle \Lambda,\rho\right\rangle }\frac{\prod
_{\alpha\in S_{0}(\Lambda)}\left\langle \alpha,\Lambda+\rho\right\rangle
}{\prod_{\alpha\in S_{0}(0)}\left\langle \alpha,\rho\right\rangle }%
\frac{(1+\zeta)^{|S_{1}(\Lambda)|}}{(1+\zeta)^{3}}.
\]
We consider the first factor%
\[
e^{\frac{2\pi i}{3}\left\langle \Lambda,\rho\right\rangle }=e^{\frac{2\pi
i}{6}\sum_{\alpha>0}\left\langle \alpha,\Lambda\right\rangle }=e^{\frac{2\pi
i}{6}\sum_{\alpha>0}\left\langle \alpha,\Lambda+\rho\right\rangle }
\]
since $\sum_{\alpha>0}\left\langle \rho,\alpha\right\rangle =2\left\langle
\rho,\rho\right\rangle =120$. Now, if $\alpha\in S_{j}(\Lambda)$ then
$\left\langle \alpha,\Lambda+\rho\right\rangle =3k_{\alpha}+j$ for $j=0,1,2$
and $k_{\alpha}=$ $\left\lfloor \frac{\left\langle \alpha,\Lambda
+\rho\right\rangle }{3}\right\rfloor $. So, if we set $\gamma=e^{\frac{\pi
i}{3}}=(1+\zeta),$ we have%
\[
e^{\frac{2\pi i}{3}\left\langle \Lambda,\rho\right\rangle }=(-1)^{\sum
_{\alpha>0}\left\lfloor \frac{\left\langle \alpha,\Lambda+\rho\right\rangle
}{3}\right\rfloor }\gamma^{|S_{1}(\Lambda)|}\gamma^{2|S_{2}(\Lambda
)|}=-(-1)^{\sum_{\alpha>0}\left\lfloor \frac{\left\langle \alpha,\Lambda
+\rho\right\rangle }{3}\right\rfloor }\gamma^{|S_{2}(\Lambda)|}.
\]
We are now ready to multiply out the formula and have%
\[
-(-1)^{\sum_{\alpha>0}\left\lfloor \frac{\left\langle \alpha,\Lambda
+\rho\right\rangle }{3}\right\rfloor }\gamma^{|S_{2}(\Lambda)|}\gamma
^{|S_{1}(\Lambda)|-3}\frac{\prod_{\alpha\in S_{0}(\Lambda)}\left\langle
\alpha,\Lambda+\rho\right\rangle }{\prod_{\alpha\in S_{0}(0)}\left\langle
\alpha,\rho\right\rangle }=
\]%
\[
-(-1)^{\sum_{\alpha>0}\left\lfloor \frac{\left\langle \alpha,\Lambda
+\rho\right\rangle }{3}\right\rfloor }\frac{\prod_{\alpha\in S_{0}(\Lambda
)}\left\langle \alpha,\Lambda+\rho\right\rangle }{2^{3}3^{9}}
\]
since $|S_{1}(\Lambda)|+|S_{2}(\Lambda)|-3=24$. We therefore have

\begin{proposition}
Assume that $\Lambda=(\lambda_{1},....,\lambda_{8},0)$ is dominant integral.
If $\sum\lambda_{i}$ is not divisible by $3$ then $\dim Hom_{SL(9,\mathbb{C)}%
}(F^{\Lambda},\mathcal{H})=0$. Let $S_{j}(\Lambda)=\{\alpha>0|\left\langle
\alpha,\Lambda+\rho\right\rangle \equiv j\operatorname{mod}3\}$ for $j=0,1,2$.
Assume $\sum\lambda_{i}\equiv0\operatorname{mod}3$ then $|S_{0}(\Lambda
)|\geq9$ and
\[
\dim Hom_{SL(9,\mathbb{C)}}(F^{\Lambda},\mathcal{H})=\left\{
\begin{array}
[c]{c}%
\frac{\dim F^{\Lambda}}{27}\text{ if }|S_{0}(\Lambda)|>9,\\
\frac{\dim F^{\Lambda}-26(-1)^{\left(  \sum_{\alpha>0}\left\lfloor
\frac{\left\langle \alpha,\Lambda+\rho\right\rangle }{3}\right\rfloor \right)
}\left(  \frac{\prod_{\alpha\in S_{0}(\Lambda)}\left\langle \alpha
,\Lambda+\rho\right\rangle }{2^{3}3^{9}}\right)  }{27}\text{ }\\
\text{if }|S_{0}(\Lambda)|=9.
\end{array}
\right.
\]

\end{proposition}

\end{document}